\documentclass[3p,numbers,review,nopreprintline]{elsarticle}
\usepackage{booktabs}
\usepackage{diagbox}
\usepackage{hyperref}
\usepackage{amsmath}
\usepackage{graphicx}
\usepackage{epstopdf}
\usepackage{subfigure}
\usepackage{rotating}
\usepackage{geometry}
\usepackage{pdflscape}
\usepackage{booktabs}
\usepackage{natbib}
\usepackage{amsfonts}
\usepackage{amsthm}
\usepackage{algorithm}
\usepackage{algorithmicx}
\usepackage{algpseudocode}
\journal{Operations Research Letters}
\newtheorem{lemma}{Lemma}
\newtheorem{theorem}{Theorem}
\newtheorem{myrule}{Rule}

\newtheorem{proposition}{Proposition}
\usepackage{threeparttable}
\usepackage{multirow}
\usepackage{slashbox}
\usepackage{pgfplots}

\begin{document}

\begin{frontmatter}

\title{A combinatorial algorithm for constrained assortment optimization under nested logit model}

\author[SHUFESIME]{Tian~Xie}\ead{xietiansh@gmail.com}
\address{Research Center for Management Science and Information Analytics, School of Information Management and Engineering, Shanghai University of Finance and Economics, China}

\begin{abstract}
We consider the assortment optimization problem with disjoint-cardinality constraints under two-level nested logit model.
To solve this problem, we first identify a candidate set with $O(mn^2)$ assortments and show that at least one optimal assortment is included in this set.
Based on this observation, a fast algorithm, which runs in $O(m n^2 \log mn)$ time, is proposed to find an optimal assortment. 
%
\end{abstract}


\end{frontmatter}

\section{Introduction}

Assortment optimization is an important topic in revenue management. 
Briefly speaking, it considers the situation that a decision maker wants to determine a set of offered products so as to maximize the expected revenue. 
In this case, each product is assigned with an exogenous fixed price and a consumer chooses at most one item from the available products 
according to some preference, which is specified by certain choice model. 
Of course, the optimal set is usually not simply the entire product set, and the optimal structure is highly dependent on the underlying choice model. 
The interested readers are referred to the excellent survey \cite{kok2015} for the related literature and applications. 

In practice, we often encounter the case that the number of offered products cannot exceed certain threshold although it is indeed beneficial to offer a larger assortment, 
due to the limited resource. For example, shelf space in a shop may be limited, or the number of advertisement displayed on a webpage is subject to the size of the screen. 
Thus, it is reasonable to consider the problem with a constraint on the total size of the offer set. 
If all the products are assumed to be the same size, that is only the number of products matters. 
Then, the constraints imposed are also referred to as cardinality constraints. 

The most popular choice model is the so-called multinomial logit (MNL) model (see \cite{mcfadden1973}). 
Talluri and van Ryzin \cite{talluri2004} first introduced this model to revenue management. 
They considered the associated unconstrained assortment optimization and showed that the optimal solution is among the assortments with \textit{revenue-ordered} structure. 
\cite{rusmevichientong2010} considered the MNL problem with cardinality constraints and 
proposed an strongly polynomial algorithm to generate an $O(nC)$-sized candidate assortment set which contains at least one optimal assortment, 
where there are $n$ products and the shop can offer no more than $C$ products. 
Although MNL model demonstrates clear advantages by providing tractability in many assortment optimization problems, it is criticized a lot due to the independence of irrelevant alternatives (IIA) property (see \cite{benakiva1985} for more information). 

The two-level nested logit model (see \cite{williams1977,mcfadden1978}) is an extension of MNL model and partially alleviates the drawbacks of IIA. 
This model divides products into several nests and describes a consumer's behavior as firstly choosing a nest and then picking a product in the chosen nest. 
\cite{davis2014} showed by imposing some additional conditions, the assortment optimization under two-level nested logit model is polynomial time solvable. 
Note that the products in the same nest must have some similarity. 
If all the products in the same nest are assumed to be the same size, 
the decision maker can plan how many products to offer for each nest before determining the optimal assortment. 
The constraints are referred as \textit{disjoint-cardinality constraints}. 
Furthermore, if we assume all the products are of the same size, the decision maker only need to plan how many products to offer in total. 
The constraints are referred as \textit{joint-cardinality constraints}. 
The linear program formulation for the unconstrained, disjoint-cardinality constrained, and joint-cardinality constrained assortment optimization problems were provided in
\cite{davis2014,gallego2014,feldman2015} respectively. By exploring more properties of the unconstrained problem,
a greedy strongly polynomial time algorithm was given in \cite{li2014}. 

In this paper, we consider the assortment optimization problem with disjoint-cardinality constraints under two-level nested logit model.
The previous work \cite{gallego2014} formulated the problem into a linear program with $(m + 1)$ decision variables and $O(m n^2)$ constraints,
However, it is well known that whether the linear programming is strongly polynomial time solvable is still open.
The major contribution of this paper is that we construct a candidate assortments set which has $O(mn^2)$ elements and contains at least one optimal assortment,
where $m$ is the number of nests and $n$ is the number of products in each nest.
Based on this, we design a combinatorial algorithm to find an optimal assortment which runs in $O(m n^2 \log mn)$ time.
To the best of our knowledge, our algorithm demonstrates the least computational complexity comparing to other known algorithms.
Our computational experiments also suggest that our approach is much more efficient than directly solving the linear program in \cite{gallego2014} by CPLEX.

We describe the problem and the linear programming formulation in Section \ref{SECT-REVIEW}. 
We derive optimality conditions and identify a candidate set with $O(mn^2)$ assortments containing at least one optimal assortment in Section \ref{SECT-OPT-CON}, 
followed by the algorithm to find an optimal assortment in Section \ref{SECT-ALG}. 
We present some computational results in Section \ref{SECT-COMPUT} to prove the efficiency of our method. 

\section{Review of the problem}\label{SECT-REVIEW}

\subsection{Problem description}


We start off by briefly describing the mathematical formulation in \cite{davis2014} of the two-level nested logit model.

Suppose there are $m$ nests indexed by $\{1, 2, \ldots, m\}$. In each nest, there are $n$ products and denote
the $j$-th product in the $i$-th nest as $ij$.
(Note that there are no overlapping between products in different nests.)
A consumer firstly chooses some nest $i$ with probability $\mathbb{P}(\mathcal{I} = i)$,
and then picks product $ij$ with probability $\mathbb{P}(\mathcal{J} = j \vert \mathcal{I} = i)$.
There is also a no-purchase option indexed by 0 which represents the case that the consumer leaves without making any purchase
and the probability is denoted as $\mathbb{P}(\mathcal{I} = 0)$.
It is assumed that the no-purchase option is isolated from the nests, i.e., if a consumer has chosen nest $i$, she will definitely buy a product.
We assign each product $ij$ with weight $v_{ij} \geq 0$ and revenue $r_{ij} \geq 0$, while the no-purchase option has weight $v_0 > 0$ and revenue 0.
In addition, each nest $i$ has a dissimilarity parameter $\gamma_i \in (0, 1]$.
Without loss of generality, we assume every nest includes exact $n$ products,
otherwise we simply add dummy products with weight 0.

An assortment is a bundle of sets $S = (S_1, S_2, \ldots, S_m)$ where $S_i \subseteq \{i1, i2, \ldots, in\}$ is the set of products that offered in nest $i$.
Therefore, the final offer set is $\cup_{i = 1}^m S_i$. Given the assortment $(S_1, S_2, \ldots, S_m)$, the customer chooses nest $i$ with probability
\[\mathbb{P}(\mathcal{I} = i) = \frac{V_i(S_i)^{\gamma_i}}{v_0 + \sum_{k = 1}^m V_k(S_k)^{\gamma_k}}, \qquad \mathbb{P}(\mathcal{I} = 0) = \frac{v_0}{v_0 + \sum_{k = 1}^m V_k(S_k)^{\gamma_k}}, \]
where $V_i(S_i) = \sum_{j \in S_i} v_{ij}$ is the total weight of products offered in nest $i$;
the customer picks product $ij$ conditioning on choosing nest $i$ with probability
\[\mathbb{P}(\mathcal{J} = j \vert \mathcal{I} = i) = \frac{v_{ij}}{V_i(S_i)}.\]
Therefore, the expected revenue is given by
\[\Pi(S_1, \ldots, S_m) = \sum_{i = 1}^m \sum_{j = 1}^n r_{ij} \mathbb{P}(\mathcal{J} = j \vert \mathcal{I} = i) \mathbb{P}(\mathcal{I} = i) = \frac{\sum_{i = 1}^m V_i(S_i)^{\gamma_i} R_i(S_i)}{v_0 + \sum_{i = 1}^m V_i(S_i)^{\gamma_i}}\]
where $R_i(S_i) = \frac{\sum_{j \in S_i} r_{ij} v_{ij}}{V_i(S_i)}$ is the weighted average revenue of products offered in nest $i$.

The goal of assortment optimization is to maximize the expected revenue. In this paper, we consider the case that $S$ has disjoint-cardinality constraints.
That is there is a size limitation $C_i$ for the offer set provided in nest $i$. Mathematically, the corresponding assortment optimization can be formulated as
\begin{equation}\label{EQN-DC}
\max_{(S_1, S_2, \ldots, S_m): \forall i, \vert S_i \vert \leq C_i} \frac{\sum_{i = 1}^m V_i(S_i)^{\gamma_i} R_i(S_i)}{v_0 + \sum_{i = 1}^m V_i(S_i)^{\gamma_i}}.
\end{equation}

\subsection{The linear programming formulation}

\cite{gallego2014} showed that the problem~\eqref{EQN-DC} can be written as a linear program. 
It is clear that for every assortment $S = (S_1, S_2, \ldots, S_m)$ and the corresponding expected revenue $Z(S)$, we have
\begin{equation}\label{EQN-EXPREV}
Z(S) = \Pi(S_1, \ldots, S_m) = \frac{\sum_{i = 1}^m V_i(S_i)^{\gamma_i} R_i(S_i)}{v_0 + \sum_{i = 1}^m V_i(S_i)^{\gamma_i}}~\Leftrightarrow~ v_0 Z = \sum_{i = 1}^m V_i(S_i)^{\gamma_i} (R_i(S_i) - Z).
\end{equation}
Therefore, for any $z > \Pi(S_1, \ldots, S_m)$, we have $v_0 z > \sum_{i = 1}^m V_i(S_i)^{\gamma_i} (R_i(S_i) - z)$;
for any $z < \Pi_i(S_1, \ldots, S_m)$, $v_0 z < \sum_{i = 1}^m V_i(S_i)^{\gamma_i} (R_i(S_i) - z)$.
Consequently, given $(S_1, S_2, \ldots, S_m)$, one has $Z(S) = \min \{ z:~ v_0 z \geq \sum_{i = 1}^m V_i(S_i)^{\gamma_i} (R_i(S_i) - z) \}$.

Denote $(S^*_1, S^*_2, \ldots, S^*_m)$ as the optimal assortment and $Z^*$ as the associated expected revenue.
Based on above discussion,
\begin{subequations}
\begin{align}
Z^* &= \max_{(S_1, S_2, \ldots, S_m): \forall i, \vert S_i \vert \leq C_i} \min \left\{ z:~ v_0 z \geq \sum_{i = 1}^m V_i(S_i)^{\gamma_i} (R_i(S_i) - z)\right\}\nonumber\\
 &= \min \left\{ z:~ v_0 z \geq \sum_{i = 1}^m V_i(S_i)^{\gamma_i} (R_i(S_i) - z) ~\forall \vert S_i \vert \leq C_i, \forall i \right\}\nonumber\\
\label{EQN-MAX} &= \min \left\{ z:~ v_0 z \geq \sum_{i = 1}^m \max_{S_i: \vert S_i \vert \leq C_i} V_i(S_i)^{\gamma_i} (R_i(S_i) - z) \right\} \\
 \label{EQN-EQVLP} &= \min \left\{ z:~ v_0 z \geq \sum_{i = 1}^m y_i, y_i \geq V_i(S_i)^{\gamma_i} (R_i(S_i) - z) ~\forall \vert S_i \vert \leq C_i, \forall i \right\}
\end{align}
\end{subequations}
Therefore, we arrive at a linear program formulation \eqref{EQN-EQVLP} with $(m + 1)$ decision variables and exponential numbers of constraints.
\cite{gallego2014} have reduced the size of constraints to polynomial, which will be discussed in Section \ref{SUBSECT-INDNEST}. 

\section{The Optimality Conditions}\label{SECT-OPT-CON}

Section \ref{SUBSECT-INDNEST} firstly describe the idea of \cite{gallego2014} that reduces exponential constraints in \eqref{EQN-EQVLP} to $O(mn^2)$ constraints, 
which is referred as \textit{optimality conditions for individual nests} in this paper. 
Based on this idea, we extend the condition to the entire system by combining all optimality conditions for individual nests. 
We construct a set of $O(mn^2)$ assortments and claim that this set contains at least one optimal assortment.
Moreover, we establish the connection between this set and a piecewise-linear function, where the root of piecewise-linear function is the maximal expected revenue.

\subsection{The optimality condition for an individual nest}\label{SUBSECT-INDNEST}

Recall that $(S^*_1, \ldots, S^*_n)$ as the optimal assortment and $Z^*$ is the maximum expected revenue. Denote $u^*_i = \max\{Z^*, \gamma_i Z^* + (1 - \gamma_i) R_i(S^*_i)\}$.
\cite{gallego2014} showed that when $\gamma_i \in (0, 1]$,
\[V_i(\hat{S}_i)(R_i(\hat{S}_i) - u^*_i) \geq V_i(S_i)(R_i(S_i) - u^*_i), \forall \vert S_i \vert \leq C_i~\Rightarrow~V_i(\hat{S}_i)^{\gamma_i} (R_i(\hat{S}_i) - Z^*) \geq V_i(S_i)^{\gamma_i} (R_i(S_i) - Z^*), \forall \vert S_i \vert \leq C_i. \]
For nest $i$, according to \eqref{EQN-MAX}, it surfies to focus on
$\max_{S_i: \vert S_i \vert \leq C_i} V_i(S_i)(R_i(S_i) - u^*_i)$.
Although the exact values of $u^*_i$ are unknown, we can still solve the following problem
\begin{equation}\label{EQN-SUBPROBLEM}
 \max_{S_i: \vert S_i \vert \leq C_i} V_i(S_i)(R_i(S_i) - u)
\end{equation}
with all $u \in \mathbf{R}$ and find the corresponding optimal solution $\hat{S}_i(u)$ (if multiple solutions exist, just pick one for each $u$ is sufficient).
The advantage of doing this is that if we set
\begin{equation}\label{EQN-CAND}
\mathcal{T}_i := \{\hat{S}_i(u), -\infty < u < \infty\},
\end{equation}
then $\hat{S}_i(u^*) \in \mathcal{T}_i$. Therefore
\begin{equation}\label{EQN-REDUCESIZE}
\max_{S_i: \vert S_i \vert \leq C_i} V_i(S_i)^{\gamma_i}(R_i(S_i) - Z^*) = \max_{S_i \in \mathcal{T}_i} V_i(S_i)^{\gamma_i}(R_i(S_i) - Z^*).
\end{equation}
We can find at least one optimal assortment $S = (S_1, \ldots, S_m)$ such that $S_i \in \mathcal{T}_i,\; \forall\; i$ and achieve the maximal expected revenue.
Thus, we name $\mathcal{T}_i$ as the \textit{candidate set} for nest $i$.


Denote $x_{ij}$ as the dummy indicator of whether ${ij} \in S_i$.
Given $u$, the problem \eqref{EQN-SUBPROBLEM} for nest $i$ can be rephrased as
\[\max \left\{\sum_{j = 1}^n v_{ij} (r_{ij} - u) x_{ij}:~ \sum_{j = 1}^n x_{ij} \leq C_i, x_{ij} \in \{0, 1\}\right\}\]
Then the optimal solution of the above problem can be specified as: $x_{ij}(u) = 1$ iff $v_{ij} (r_{ij} - u) \geq 0$
and $v_{ij} (r_{ij} - u)$ is in the largest $C_i$ elements among $\{v_{ij} (r_{ij} - u)\}_{j = 1}^n$; $x_{ij}(u) = 0$ otherwise.
Note that if ties exist, we can rank them arbitrarily.
Then a candidate for nest $i$ can be constructed as: $\hat{S}_i(u) = \{ij: x_{ij}(u) = 1\}$.

Next we discuss the size of $\mathcal{T}_i$.
Consider the following $n + 1$ lines: $f_0(u) = 0$, $f_j(u) = v_{ij} (r_{ij} - u)$.
To construct $\hat{S}_i(u)$, we can simply choose $C_i$ lines with highest $f_j(u)$, and drop all lines below $f_0(u)$.
There are $q \leq \frac{1}{2} n(n + 1)$ crosspoints for those $(n + 1)$ lines,
and the corresponding horizontal coordinate is denoted as $I_1 \leq I_2 \leq \ldots \leq I_q$.
For any consecutive $I_k < I_{k + 1}$ with different horizontal coordinate,
the relative order of all lines don't change between $u \in (I_k, I_{k + 1})$
because for any $j', j''$, the relative order $f_{j'}(u)$ and $f_{j''}(u)$ don't change between $u \in (I_k, I_{k + 1})$.
There could be ties on some $u = I_k$, and either $u = I_k - \varepsilon$ or $u = I_k + \varepsilon$ ($\varepsilon \rightarrow 0$) can be the alternative tie-breaking rule.
Therefore, a set $\mathcal{T}_i$ with no more than $O(n^2)$ candidates is sufficient for any nest $i$. Note that now \eqref{EQN-EQVLP} can reduced to 
\begin{equation}\label{EQN-REDUCEDLP}
Z^* = \min \left\{ z \middle\vert v_0 z \geq \sum_{i = 1}^m y_i, y_i \geq V_i(S_i)^{\gamma_i} (R_i(S_i) - z) ~\forall S_i \in \mathcal{T}_i, \forall i \right\}
\end{equation}
which is a linear program with $(m + 1)$ decision variables and $O(mn^2)$ constraints. We refer interested readers to \cite{gallego2014} for more details.

\subsection{The optimality condition of the whole problem}\label{SUBSECT-ENTIRE}

To determine the exact optimal solution, the characterization of all the optimal assortments is stated in Proposition \ref{PROP-OPT-ASSORT}.
\begin{proposition}[Optimal Assortments]\label{PROP-OPT-ASSORT}
Denote $Z^*$ as the maximal expected revenue. An assortment $S = (S_1, \ldots, S_m)$ is an optimal assortment if and only if 
\[S_i \in \arg\max_{S_i: \vert S_i \vert \leq C_i}\{V_i(S_i)^{\gamma_i}(R_i(S_i) - Z^*)\}, \forall i.\]
\end{proposition}
\begin{proof}
Since $Z^*$ is the maximal expected revenue, by the definition of expected revenue, it must be the case that
\[Z(S) - Z^* = \frac{\sum_{i = 1}^m V_i(S_i)^{\gamma_i} (R_i(S_i) - Z^*) - v_0 Z^*}{v_0 + \sum_{i = 1}^m V_i(S_i)^{\gamma_i}} \leq 0, \]
which is equivalent to $\sum_{i = 1}^m V_i(S_i)^{\gamma_i} (R_i(S_i) - Z^*) - v_0 Z^* \leq 0$.
For any optimal assortment $S^* = (S^*_1, \ldots, S^*_m)$, $Z(S^*) - Z^* = 0$, which is equivalent to $\sum_{i = 1}^m V_i(S^*_i)^{\gamma_i} (R_i(S^*_i) - Z^*) - v_0 Z^* = 0$.
If $S_i$ maximizes $V_i(S_i)^{\gamma_i}(R_i(S_i) - Z^*)$ for all $i$,
\[0 \geq \sum_{i = 1}^m V_i(S_i)^{\gamma_i} (R_i(S_i) - Z^*) - v_0 Z^* \geq \sum_{i = 1}^m V_i(S^*_i)^{\gamma_i} (R_i(S^*_i) - Z^*) - v_0 Z^* = 0, \]
which implies that $S$ is also an optimal assortment.
Otherwise, if there exists some $j$ that $V_j(S_j)^{\gamma_j}(R_j(S_j) - Z^*) < V_j(S'_j)^{\gamma_j}(R_j(S'_j) - Z^*)$, then
\[\sum_{i = 1}^m V_i(S_i)^{\gamma_i} (R_i(S_i) - Z^*) - v_0 Z^* < V_i(S'_i)^{\gamma_i} (R_i(S'_i) - Z^*) + \sum_{i \neq j} V_i(S_i)^{\gamma_i} (R_i(S_i) - Z^*) - v_0 Z^* \leq 0, \]
which implies that $S$ is not optimal.
\end{proof}

Note that optimal assortment may not be unique. In some extreme cases, the optimal set can include exponential number of elements.
However, to find one optimal assortment is easy. In particular,
\eqref{EQN-REDUCESIZE} and Proposition \ref{PROP-OPT-ASSORT} show that the assortment $(S_1, \ldots, S_m)$ where $S_i \in \arg\max_{S_i \in \mathcal{T}_i}\{V_i(S_i)^{\gamma_i}(R_i(S_i) - Z^*)\}$ is optimal.
The difficulty is that the exact value of $Z^*$ is still unknown.
Next, we want to find $Z^*$ through the linear program formulation.

\begin{proposition}[Necessary condition]\label{PROP-OPTEQ}
All optimal solutions $(Z^*, \{y^*_1, \ldots, y^*_m\})$ for \eqref{EQN-EQVLP} satisfy
\begin{align*}
&v_0 Z^* = \sum_{i = 1}^m y^*_i\\
&y^*_i = V_i(S^*_i)^{\gamma_i} (R_i(S^*_i) - Z^*) \geq V_i(S_i)^{\gamma_i} (R_i(S_i) - Z^*), \forall \vert S_i \vert \leq C_i, \forall i
\end{align*}
for some $(S^*_1, \ldots, S^*_m)$ such that $\vert S_i \vert \leq C_i, \forall i$. 
\end{proposition}
\begin{proof}
Prove by contradiction. Suppose there is an optimal solution $(Z^*, \{y^*_1, \ldots, y^*_m\})$ where for some $i$,
$y^*_i = V_i(S^*_i)^{\gamma_i} (R_i(S^*_i) - Z^*) \geq V_i(S_i)^{\gamma_i} (R_i(S_i) - Z^*), \forall S_i \in \mathcal{T}_i, $
but $v_0 Z^* > \sum_{i = 1}^m y^*_i$. Plug $y^*_i$ in, we have
$v_0 Z^* > \sum_{i = 1}^m V_i(S^*_i)^{\gamma_i} (R_i(S^*_i) - Z^*)$.
Recall that $v_0 > 0$ and all $v_{ij}$ are non-negative. There always exists $\varepsilon > 0$ such that
$v_0 (Z^* - \varepsilon) \geq \sum_{i = 1}^m V_i(S^*_i)^{\gamma_i} (R_i(S^*_i) - Z^* + \varepsilon)$
and hence it is possible to reduce $Z^*$ to $Z^* - \varepsilon$. The solution is not optimal.

Suppose there is an optimal solution $(Z^*, \{y^*_1, \ldots, y^*_m\})$ where for some $i$,
$y^*_i > V_i(S_i)^{\gamma_i} (R_i(S_i) - Z^*), \forall S_i \in \mathcal{T}_i$,
then reducing the corresponding $y'^*_i$ to $\max_{S_i \in \mathcal{T}_i} V_i(S_i)^{\gamma_i} (R_i(S_i) - Z^*)$ will also make
$(Z^*, \{y^*_1, \ldots, y^*_{i - 1}, y'^*_i, y^*_{i + 1}, y^*_m\})$ feasible and hence optimal.
For this case, $v_0 Z^* > \sum_{i = 1}^m y'^*_i$. In the case above, the solution is not optimal.
\end{proof}

Proposition \ref{PROP-OPTEQ} shows that the optimal objective for \eqref{EQN-EQVLP} must belong to the set
\[\mathcal{Z}^* = \left\{z:~ v_0 z = \sum_{i = 1}^m \max_{S_i:~ \vert S_i \vert \leq C_i} V_i(S_i)^{\gamma_i} (R_i(S_i) - z)\right\}.\]
If we construct $\{\mathcal{T}_i\}_{i = 1, \ldots, m}$ in the form of \eqref{EQN-CAND}, by \eqref{EQN-REDUCESIZE} the set is equivalent to
\[\mathcal{Z}^* = \left\{z:~ v_0 z = \sum_{i = 1}^m \max_{S_i \in \mathcal{T}_i} V_i(S_i)^{\gamma_i} (R_i(S_i) - z)\right\}.\]

For each nest $i$, we focus on the following function
\[g_i(z) = \max_{S_i \in \mathcal{T}_i} V_i(S_i)^{\gamma_i}(R_i(S_i) - z).\]
Obviously, the epigraph of $g_i(\cdot)$ is a intersection of $\vert \mathcal{T}_i \vert$ halfplanes,
so $g_i(\cdot)$ is continuous, decreasing (all $V_i(S_i) \geq 0$), convex, piecewise-linear and has at most $\vert \mathcal{T}_i \vert - 1$ breakpoints.
Summing up $\{g_i(\cdot)\}_{i = 1}^m$ and $-v_0 z$ yields
\[G(z) = -v_0 z + \sum_{i = 1}^m \max_{S_i \in \mathcal{T}_i} V_i(S_i)^{\gamma_i}(R_i(S_i) - z).\]
\begin{lemma}\label{LEMMA-MINCONV}
$G(z)$ is a strictly decreasing piecewise-linear convex funtion on $z$ with at most $\sum_{i = 1}^m (\vert \mathcal{T}_i \vert - 1) = O(mn^2)$ breakpoints.
\end{lemma}
\begin{proof}
Firstly, $g_0(z) := -v_0 z$ is a strictly decreasing linear function. For $i = 1, 2, \ldots, n$,
$g_i(z)$ is a continuous piecewise-linear convex function with at most $\vert \mathcal{T}_i \vert - 1$ breakpoints.
Between any two consecutive breakpoints, $g_i(z) = V_i(S_i)^{\gamma_i}R_i(S_i) - V_i(S_i)^{\gamma_i} z$ is decreasing on $z$ for all $i$.
In summary, $G(z)$ is strictly decreasing on $\mathbf{R}$.
Therefore, $G(z) = \sum_{i = 0}^m g_i(z)$ is a strictly decreasing piecewise-linear convex funtion on $z$ with at most $\sum_{i = 1}^m (\vert \mathcal{T}_i \vert - 1)$ breakpoints.
\end{proof}

\begin{proposition}[Necessary condition is sufficient]\label{PROP-SINGLETON}
$\mathcal{Z}^*$ is a singleton.
\end{proposition}
\begin{proof}
$\mathcal{Z}^*$ is essentially $\{z:~ G(z) = 0\}$.
Since $G(z)$ is strictly decreasing while $G(-\infty) = \infty$ and $G(\infty) = -\infty$,
there must be a unique $z$ such that $G(z) = 0$, which implies that $\mathcal{Z}^*$ is a singleton.
\end{proof}

By Proposition \ref{PROP-SINGLETON}, $G(z) = 0$ uniquely defines $\mathcal{Z}^*$, so $Z^*$ should be the unique element in $\mathcal{Z}^*$.
The analysis above shows that $(S_1, \ldots, S_m)$ where $S_i \in \arg\max_{S_i \in \mathcal{T}_i}\{V_i(S_i)^{\gamma_i}(R_i(S_i) - Z^*)\}$ is an optimal assortment.

Before ending this section, we show that our optimality condition can be explained as the existence of a collection of $O(mn^2)$ candidate assortments that contains at least one optimal assortment. To this end,
denote
\begin{equation}\label{EQN-S}
S(z) = (S_1(z), \ldots, S_m(z))
\end{equation}
where $S_i(z) = \arg\max_{S_i \in \mathcal{T}_i} V_i(S_i)^{\gamma_i}(R_i(S_i) - z)$. 
If multiple maximizers, arbitrarily pick $S_i(z - \varepsilon)$ or $S_i(z + \varepsilon)$.
\begin{proposition}[Candidates for the whole problem]\label{PROP-GLOBAL-CAND}
There exists a set of assortments $\mathcal{T} = \{S(z), -\infty < z < \infty\}$
which contains at least one optimal assortment, and $\vert \mathcal{T} \vert = O(mn^2)$.
\end{proposition}
\begin{proof}
Firstly, $S(Z^*) = (S_1(Z^*), \ldots, S_m(Z^*)) \in \mathcal{T}$.
For each nest $i$, from the fact 
\[S_i(Z^*) \in \arg\max_{S_i \in \mathcal{T}_i}\{V_i(S_i)^{\gamma_i}(R_i(S_i) - Z^*),\]
and equivalent relation \eqref{EQN-REDUCESIZE}, it follows that $S_i(Z^*) \in \arg\max_{S_i: \vert S_i \vert \leq C_i}\{V_i(S_i)^{\gamma_i}(R_i(S_i) - Z^*)$.
Then by Proposition \ref{PROP-OPT-ASSORT}, $S(Z^*)$ is an optimal assortment which belongs to $\mathcal{T}$.

Next, we analyze the size of $\mathcal{T}$.
For each nest $i$, $S_i(\cdot)$ is the maximizer of $g_i(\cdot)$ which is piecewise with at most $\vert \mathcal{T}_i \vert - 1$ breakpoints.
The union of $S_i(\cdot)$'s breakpoints will not exceed $\sum_{i = 1}^m (\vert \mathcal{T}_i \vert - 1) = O(mn^2)$.
It is obvious that between any adjacent breakpoints, $S(\cdot)$ is consistent.
Therefore, $\mathcal{T}$ will contain $O(mn^2)$ elements.
\end{proof}
We can see that this result is consistent with Lemma \ref{LEMMA-MINCONV}:
each linear piece of $G(\cdot)$ corresponds to an assortment and $G(z)$ is associated with $S(z)$ as the maximizer for all $z \in \mathbf{R}$.
In next section, we will design an algorithm to find an optimal assortment by enumerating all elements in $\mathcal{T}$ efficiently.

\section{The algorithm and analysis}\label{SECT-ALG}

Our algorithm is comprised of two stages.
The first is to generate $\{\mathcal{T}_i\}_{i = 1}^m$: the candidate sets for each nest as described in Section \ref{SUBSECT-INDNEST}.
Then we enumerate all the $O(mn^2)$ elements in $\mathcal{T}$, which is given in Proposition \ref{PROP-GLOBAL-CAND}, 
to find an optimal assortment.

\subsection{Generating the candidate assortments for an individual nest}\label{SECT-CANDIDATE}

Here we propose a simple algorithm to generate $\mathcal{T}_i$ for nest $i$, 
which is essentially determine the topmost $C_i$ lines in $\{f_j(u)\}_{j = 1}^n$ for all $u \in \mathbf{R}$, where $f_j(u) = v_{ij} (r_{ij} - u)$. 

Without loss of generality, we assume $v_{i1} \leq v_{i2} \leq \ldots \leq v_{in}$; 
if some $v_{ij} = v_{i(j + 1)}$, then order $v_{ij} r_{ij} \leq v_{i(j + 1)} r_{i(j + 1)}$. 
We scan from $u \rightarrow -\infty$ to $\infty$. 
Initially, for $u \rightarrow -\infty$, the order from highest to lowest must be $(n, n - 1, \ldots, 1, 0)$. 

As described in Section \ref{SUBSECT-INDNEST}, the relative order of lines will be consistent between two consecutive crosspoints. 
If there is a crosspoint $I_k$ only for two lines $j_1 > j_2$, i.e., $f_{j_1}(I_k) = f_{j_2}(I_k)$, 
their order must be consecutive and should be swapped before and after $I_k$,
i.e., $f_{j_1}(I_k - \varepsilon) > f_{j_2}(I_k - \varepsilon)$ and $f_{j_1}(I_k + \varepsilon) < f_{j_2}(I_k + \varepsilon)$.
When $r \geq 2$ lines $\{j_1, j_2, \ldots, j_r\}$ share the same crosspoint $I_k$,
i.e., $f_{j_1}(I_k) = f_{j_2}(I_k) = \ldots = f_{j_r}(I_k)$, their rank must be consecutive and should be reversed before and after $I_k$ as $u$ increases.
Without loss of genreality, assume $j_1 > \ldots > j_r$, which implies $f_{j_1}(I_k - \varepsilon) > \ldots > f_{j_r}(I_k - \varepsilon)$.
The common crosspoint can be regarded as $\frac{1}{2}r(r - 1)$ overlapped crosspoints. 
After a series of $\frac{1}{2}r(r - 1)$ swap operations:
\[(j_1, j_2), (j_1, j_3), \ldots, (j_1, j_r), (j_2, j_3), \ldots, (j_2, j_r), \ldots, (j_{r - 1}, j_r),\]
the order of $(j_1, \ldots, j_r)$ is reversed, which exactly matches $f_{j_r}(I_k + \varepsilon) > \ldots > f_{j_1}(I_k + \varepsilon)$.
Rule \ref{RULE-TIEBREAK} is a summarization. 

As an illustration, Figure \ref{FIG-SWAP} consists of 5 lines $\{1, 2, 3, 4, 5\}$ ranked from lowest absolute slope to highest. 
There are 4 crosspoints at $u = 0.5$ (3 ones are overlapped). For $u = 0.5 - \varepsilon$, the order is $(5, 3, 2, 4, 1)$; 
after a series of swaption: $(3, 5), (2, 5), (2, 3), (1, 4)$, the order becomes $(2, 3, 5, 1, 4)$, which is exactly the order at $u = 0.5 + \varepsilon$. 

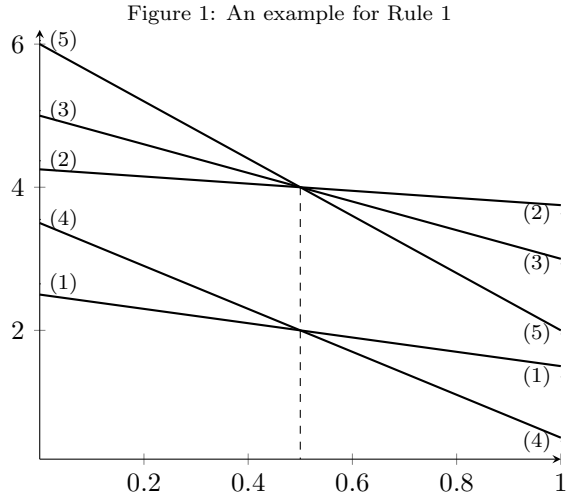
\begin{figure}[h]
\caption{An example for Rule \ref{RULE-TIEBREAK}}
\label{FIG-SWAP}
\centering
\begin{tikzpicture}
    \begin{axis}[axis x line=middle,
                 axis y line=middle,
				 xmin=0, xmax=1, ymin=0.2, ymax=6.2]
        \addplot[thick,domain=0:1]{(x-0.5)*(-4)+4};
		\addplot[thick,domain=0:1]{(x-0.5)*(-3)+2};
		\addplot[thick,domain=0:1]{(x-0.5)*(-2)+4};
		\addplot[thick,domain=0:1]{(x-0.5)*(-0.5)+4};
		\addplot[thick,domain=0:1]{(x-0.5)*(-1)+2};
		\draw (axis cs:0,6.05)--(axis cs:0,6.05) node[right] {\footnotesize{$(5)$}};
		\draw (axis cs:0,3.55)--(axis cs:0,3.55) node[right] {\footnotesize{$(4)$}};
		\draw (axis cs:0,5.07)--(axis cs:0,5.07) node[right] {\footnotesize{$(3)$}};
		\draw (axis cs:0,4.37)--(axis cs:0,4.37) node[right] {\footnotesize{$(2)$}};
		\draw (axis cs:0,2.65)--(axis cs:0,2.65) node[right] {\footnotesize{$(1)$}};
		
		\draw (axis cs:1,1.95)--(axis cs:1,1.95) node[left] {\footnotesize{$(5)$}};
		\draw (axis cs:1,0.45)--(axis cs:1,0.45) node[left] {\footnotesize{$(4)$}};
		\draw (axis cs:1,2.93)--(axis cs:1,2.93) node[left] {\footnotesize{$(3)$}};
		\draw (axis cs:1,3.63)--(axis cs:1,3.63) node[left] {\footnotesize{$(2)$}};
		\draw (axis cs:1,1.35)--(axis cs:1,1.35) node[left] {\footnotesize{$(1)$}};
		
		\draw [style=dashed] (axis cs:0.5,4)--(axis cs:0.5,0.2);
    \end{axis}
\end{tikzpicture}
\end{figure}

\begin{myrule}[Tie-breaking]\label{RULE-TIEBREAK}
For any pair of crosspoints $(u_k, j'_k, j''_k)$ and $(u_{k + 1}, j'_{k + 1}, j''_{k + 1})$, $u_k \leq u_{k + 1}$; 
if $u_k = u_{k + 1}$, then order $j'_k \geq j'_{k + 1}$; if $u_k = u_{k + 1}$ and $j'_k = j'_{k + 1}$, then order $j''_k \geq j''_{k + 1}$.
\end{myrule}
That is to say, since there are at most $O(n^2)$ crosspoints, it is sufficient to perform at most $O(n^2)$ swaps as $u$ increases.
Once the $C_i$-th and $(C_i + 1)$-th elements swap, we update $\hat{S}_i(u)$ with the corresponding $V_i(\hat{S}_i(u))$ and $R_i(\hat{S}_i(u))$ by removing one and inserting another.
Each operation can be done in $O(1)$ time.

\begin{algorithm}[h]
\caption{Generating candidate set $\mathcal{T}_i$ for nest $i$}\label{ALG-GENCAN}
\begin{algorithmic}[1]
\Require $v_{ij}, r_{ij}, \forall j$; $C_i$
\Ensure $\mathcal{T}_i$
\State Reorder the products such that $v_{i1} \leq v_{i2} \leq \ldots \leq v_{in}$;
if $v_{ij} = v_{i(j + 1)}$, then order $v_{ij} r_{ij} \leq v_{i(j + 1)} r_{i(j + 1)}$. 
\State Define an order $R \leftarrow (n, n - 1, \ldots, 0)$. 
\State Denote $\hat{S}_i$ as the first $C_i$ elements of $R$ with $V_i(\hat{S}_i)$ and $R_i(\hat{S}_i)$, $\mathcal{T}_i \leftarrow \{\hat{S}_i\}$. 
\State Generate crosspoints $\mathcal{I} = \{(u_k, j'_k, j''_k): f_{j'}(u) = f_{j''}(u), j'_k < j''_k\}$ and sort them by Rule \ref{RULE-TIEBREAK}. 
\For{$k = 1, \ldots, p$}
	\State Swap the order of $(j'_k, j''_k)$ in $R$. 
	\If{the elements at $C_i$-th and $C_{i + 1}$-th order change after $u_k$}
		\State Denote $\hat{S}_i$ as the first $C_i$ elements of $R$ (if $\hat{S}_i$ contains 0, ignore 0 and drop all the elements below). 
		\State Update $V_i(\hat{S}_i)$ and $R_i(\hat{S}_i)$ by replacing the corresponding elements. 
		\State $\mathcal{T}_i \leftarrow \mathcal{T}_i \bigcup \{\hat{S}_i\}$. 
	\EndIf
\EndFor
\State \Return $\mathcal{T}_i$.
\end{algorithmic}
\end{algorithm}

Note that it is not necessary to record all the $\hat{S}_i(u)$ because $V_i(\hat{S}_i(u))$ and $R_i(\hat{S}_i(u))$ are sufficient. 
Generating and sorting $O(n^2)$ crosspoints costs $O(n^2 \log n)$, and generating all $V_i(\hat{S}_i(u))$ and $R_i(\hat{S}_i(u))$ by swapping at most $O(n^2)$ pairs of elements costs $O(n^2)$.
Theorem \ref{THM-GENCAN} is a summarization of the above analysis. 

\begin{theorem}\label{THM-GENCAN}
If we record $V_i(\hat{S}_i(u))$ and $R_i(\hat{S}_i(u))$ instead of each assortment $\hat{S}_i(u)$, for each nest $i$, Algorithm \ref{ALG-GENCAN} terminates in $O(n^2 \log n)$ time with candidate set $\mathcal{T}_i$.
\end{theorem}

\subsection{Enumerate the candidate assortments for the entire system}\label{SECT-PIECELINE}

Here we propose an strongly polynomial algorithm described in Algorithm \ref{ALG-DISJOINT} to solve the problem. 
Briefly, it is essentially enumerating the candidate set $\mathcal{T}$ in Proposition \ref{PROP-GLOBAL-CAND}.

\begin{algorithm}[h]
\caption{Combinatorial algorithm for disjoint-cardinality constraints}
\label{ALG-DISJOINT}
\begin{algorithmic}[1]
\Require $\left\{\mathcal{T}_i\right\}_{i = 1}^m$ generated by Algorithm \ref{ALG-GENCAN}.
\Ensure Optimal expected revenue $Z^*$ and optimal assortment $(S_1^*, S_2^*, \ldots, S_m^*)$.
\State \textit{Step 1: Iteratively calculating the expression of $g_i(\cdot)$}
\For{$i = 1, \ldots, m$}
	\State Let $\mathcal{L}^i = \{(a^i_j, b^i_j) := (V_i(S_i)^{\gamma_i} R_i(S_i), V_i(S_i)^{\gamma_i}), \forall S_i \in \mathcal{T}_i\}$. 
	\State Sort $\mathcal{L}^i$ to ensure $b^i_j > b^i_{j + 1}, \forall j$ while drop all $(a^i_{j'}, b^i_{j'})$ such that $\exists b^i_{j'} = b^i_{j''}$ and $a^i_{j'} < a^i_{j''}$. 
	\State $q^i \leftarrow 1$, $(\tilde{u}^i_1, \tilde{a}^i_1, \tilde{b}^i_1) \leftarrow (-\infty, a^i_1, b^i_1)$. 
	\For{$j = 2, \ldots, \vert \mathcal{L}^i \vert$}
		\While{$a^i_j - b^i_j \tilde{u}^i_{q^i} \geq \tilde{a}^i_{q^i} - \tilde{b}^i_{q^i} \tilde{u}^i_{q^i}$}
			\State $q^i \leftarrow q^i - 1$. 
		\EndWhile
		\State $q^i \leftarrow q^i + 1$, $(\tilde{u}^i_q, \tilde{a}^i_q, \tilde{b}^i_q) \leftarrow \left(\frac{\tilde{a}^i_q - a^i_j}{\tilde{b}^i_q - b^i_j}, a^i_j, b^i_j\right)$. 
	\EndFor
\EndFor
\State \textit{Step 2: Enumerating $\mathcal{T}$ to find the maximal expected revenue}
\State $A \leftarrow \sum_{i = 1}^m \tilde{a}^i_1, B \leftarrow v_0 + \sum_{i = 1}^m \tilde{b}^i_1$. 
\State $\Delta \leftarrow \bigcup_{i = 1}^m \left\{(u, \Delta a, \Delta b) := (\tilde{u}^i_j, \tilde{a}^i_j - \tilde{a}^i_{j - 1}, \tilde{b}^i_j - \tilde{b}^i_{j - 1}), \forall j = 2, \ldots, q^i\right\}$, sort $\Delta$ to ensure $u_i \leq u_{i + 1}, \forall i$. 
\State $Z^* \leftarrow A / B$. 
\For{$i = 1, \ldots, \vert \Delta \vert$}
	\State $A \leftarrow A + (\Delta a)_i, B \leftarrow B + (\Delta b)_i$. 
	\State $Z^* \leftarrow \max\{Z^*, A / B\}$. 
\EndFor
\State $S_i^* = \arg \max_{S_i \in \mathcal{T}_i} V_i(S_i)^{\gamma_i}(R_i(S_i) - Z^*)$. 
\State \Return Optimal expected revenue $Z^*$ and optimal assortment $(S_1^*, S_2^*, \ldots, S_m^*)$.
\end{algorithmic}
\end{algorithm}

Each loop in Step 1 calculates the expression of $g_i(\cdot)$ which is stored as 
$g_i(z) = \tilde{a}^i_j - \tilde{b}^i_j z, \tilde{u}^i_j \leq z \leq \tilde{u}^i_{j + 1}, j = 1, \ldots, q$ where $\tilde{u}^i_{q + 1} = \infty$. 

By definition, $g_i(z)$ is a maximization of a sequence of linear functions simplified as $g_i(z) := \max_{j = 1}^{\vert \mathcal{L} \vert} \{a^i_j - b^i_j z\}$. 
Without loss of generality, we assume $b^i_1 > \ldots > b^i_{\vert \mathcal{L} \vert}$ (ranking the lines from highest steepness to lowest) 
and no $(a^i_{j'}, b^i_{j'})$ such that $\exists b^i_{j'} = b^i_{j''}$ and $a^i_{j'} < a^i_{j''}$ (line $j'$ is redundant). 
With a slight abuse of notation, we define $g_i^k(z) = \max_{j = 1}^k \{a^i_j - b^i_j z\}$. 
For $k \geq 2$, the recursive formulation is $g_i^k(z) = \max\{g_i^{k - 1}(z), a^i_k - b^i_k z\}$. 
The iteration ends up with $g_i(z) = g_i^{\vert \mathcal{L} \vert}(z)$. 

For $k \geq 2$, we know that the (sub)gradient of $g_i^{k - 1}(z)$ for any $z \in \mathbf{R}$ is strictly less than $-b^i_k$. 
There must be a breakpoint $u$ such that $g_i^{k - 1}(u) = a^i_k - b^i_k u$, 
$z < u \Leftrightarrow g_i^{k - 1}(z) > a^i_k - b^i_k z$ and $z > u \Leftrightarrow g_i^{k - 1}(z) < a^i_k - b^i_k z$ (see Figure \ref{FIG-ADDLINE} for an illustration). 
Therefore $g_i^k(z)$ can be formulated by some leftmost linear piece of $g_i^{k - 1}(\cdot)$ plus a new linear piece 
\[g_i^k(z) = \left\{
\begin{array}{ll}
g_i^{k - 1}(z), & z \leq u\\
a^i_k - b^i_k z, & z \geq u
\end{array}
\right.\]
After repeatedly delete the rightmost linear piece until we find the intersection point $u$, 
we create a new linear piece for $z \geq u$ and the function is then updated to $g_i^k(\cdot)$. 

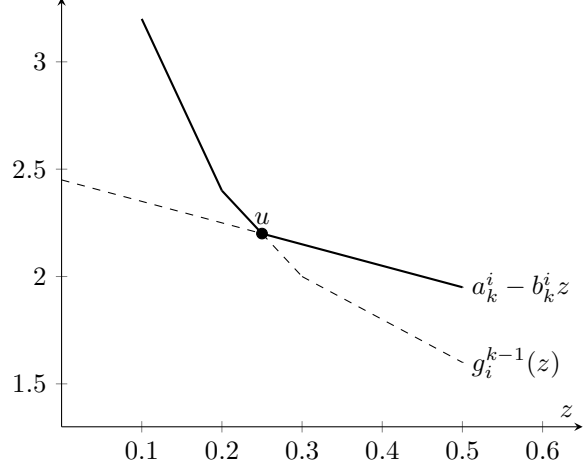
\begin{figure}[h]
\caption{The maximization of $g_i^{k - 1}(z)$ and $a^i_k - b^i_k z$}
\label{FIG-ADDLINE}
\centering
\begin{tikzpicture}
    \begin{axis}[axis x line=middle,
                 axis y line=middle,
                 xlabel=$z$,
				 xmin=0, xmax=0.65,
				 ymin=1.3, ymax=3.3]
        \addplot[thick,domain=0.1:0.2]{4-8*x};
		\addplot[thick,domain=0.2:0.25]{3.2-4*x};
		\addplot[dashed,domain=0.25:0.3]{3.2-4*x};
		\addplot[dashed,domain=0.3:0.5]{2.6-2*x};
		\addplot[dashed,domain=0.0:0.25]{2.45-1*x};
		\addplot[thick,domain=0.25:0.5]{2.45-1*x};
		\addplot[mark=*] coordinates {(0.25,2.2)};
		\draw (axis cs:0.25,2.2)--(axis cs:0.25,2.2) node[right,above] {$u$};
		\draw (axis cs:0.5,1.95)--(axis cs:0.5,1.95) node[right] {$a^i_k - b^i_k z$};
		\draw (axis cs:0.5,1.6)--(axis cs:0.5,1.6) node[right] {$g_i^{k - 1}(z)$};
    \end{axis}
\end{tikzpicture}
\end{figure}

In Step 2, we enumerate all assortments in $\mathcal{T}$ to find the maximal expected revenue. 
Denote $A(z) = \sum_{i = 1}^m V_i(S_i(z))^{\gamma_i} R_i(S_i(z))$ and $B(z) = v_0 + \sum_{i = 1}^m V_i(S_i(z))^{\gamma_i}$. 
By definition, the expected revenue for assortment $S_i(z)$ is $A(z) / B(z)$. 
Also, $G(z) = v_0 z + \sum_{i = 1}^m g_i(z) := A(z) - B(z) z$.
To calculate the maximal expected revenue, the algorithm only needs attributes $A(z)$ and $B(z)$ for each $S(z)$. 

The set of $G(\cdot)$'s breakpoints consists of all breakpoints of $g_i(\cdot)$ for all $i = 1, \ldots, m$;
also $A(\cdot)$ and $B(\cdot)$. The iterative expression for $A(\cdot)$ is 
\begin{align*}
A\left(-\infty\right) &= A\left(-\infty\right) + \sum_{i = 1}^m \tilde{a}^i_1, \\
A\left(z_+\right) &= A\left(z_-\right) + \sum_{(i, j): \tilde{u}^i_j = z} (\tilde{a}^i_j - \tilde{a}^i_{j - 1}), 
\end{align*}
and $B(\cdot)$ is similar. 
We gather and sort all breakpoints to $u_1 \leq \ldots \leq u_{\vert \Delta \vert}$. 
For each $u_i$, we update $(A, B)$ to $(A(u_i), B(u_i))$ from previous one. 
The algorithm finds the largest expected revenue $Z^* = \max_{z \in \mathbf{R}}\{A(z) / B(z)\}$ by essentially enumerating all elements in $\mathcal{T}$. 
By the way, it is simple to find an optimal assortment given $Z^*$. 

Next we discuss the time complexity except for generating $\{\mathcal{T}_i\}_{i = 1}^m$.

The first part: For each nest $i$, sort $\vert \mathcal{L}^i \vert = O(n^2)$ elements costs $O(n^2 \log n)$; 
each iteration will create a new piece and each linear piece can be delete at most once. 
In total, the first part costs $O(mn^2 \log n)$.

The second part: the sorting, which is essentially merging $m$ $O(n^2)$-sized ordered lists, can be done in $O(mn^2 \log m)$; 
the enumeration costs $O(mn^2)$. 

Theorem \ref{THM-DISJOINT} is a summary of the above analysis.

\begin{theorem}\label{THM-DISJOINT}
Algorithm \ref{ALG-DISJOINT} terminates with an optimal assortment in $O(mn^2 \log m)$ time. 
\end{theorem}

The workflow to solve the optimization problem is: (1) Generate $\{\mathcal{T}_i\}_{i = 1}^m$ by running Algorithm \ref{ALG-GENCAN} $m$ times, which costs $O(mn^2 \log n)$ time; 
(2) Run Algorithm \ref{ALG-DISJOINT} to determine the maximal expected revenue and find an optimal assortment. 
The total running time is $O(mn^2 \log mn)$. 

\section{Computational Experiments}\label{SECT-COMPUT}

In this section, we conduct some computational experiments to show the efficiency of our algorithm.
In particular, we compare Algorithm \ref{ALG-DISJOINT} with directly solving linear program formulation \eqref{EQN-REDUCEDLP}.
In our tests, it is solved by IBM ILOG CPLEX 12.6 (64-bit version) via the ILOG Concert API.
The base frequency of the CPU for the tests is 2.10GHz.

We randomly generate 10 problem instances for each setting.
The number of nest $m$ and the products in each nest $n$ is specified in each setting.
For all product $ij$, the weight $v_{ij}$ is i.i.d. random variable uniformly from [0.1, 10.0] and the revenue $r_{ij}$ is i.i.d. uniformly from [0, 10].
While the weight of no-purchase $v_0 = 1$.
The dissimilarity parameters $\gamma_i = 0.5$ for all $i$  and the constraints $C_i = n / 2$ for $i$.

The results are shown in Table \ref{TABLE-COMPARE}. Note that the experiment is memory-consuming.
Our Algorithm \ref{ALG-DISJOINT} is optimized to reduce the memory consumption and therefore performs well for all the instances.
In particular, Table \ref{TABLE-COMPARE} suggests that Algorithm \ref{ALG-DISJOINT} is able to solve a instance with $200000$ nests and each nest includes $200$ products within a minute. 
By contrast, CPLEX comsumes more time and more memory when the problem size increases.
In summary, the results suggest that our algorithm is an efficient approach when the problem size is large.

\begin{table}[h]
\caption{Comparison of the running time in seconds for Algorithm \ref{ALG-DISJOINT} and CPLEX}\label{TABLE-COMPARE}
\centering
\begin{threeparttable}
\begin{tabular}{|r|rr|rr|rr|rr|rr|}
\hline
\multirow{2}{*}{\backslashbox{$m$}{$n$}} & \multicolumn{2}{c|}{10} & \multicolumn{2}{c|}{20} & \multicolumn{2}{c|}{50} & \multicolumn{2}{c|}{100} & \multicolumn{2}{c|}{200}\\
\cline{2-11}
& Alg. \ref{ALG-DISJOINT} & CPLEX & Alg. \ref{ALG-DISJOINT} & CPLEX & Alg. \ref{ALG-DISJOINT} & CPLEX & Alg. \ref{ALG-DISJOINT} & CPLEX & Alg. \ref{ALG-DISJOINT} & CPLEX\\
\hline
100&	0.00&	0.01&	0.00&	0.02&	0.01&	0.06&	0.00&	0.06&	0.01&	0.11\\
200&	0.00&	0.03&	0.00&	0.06&	0.01&	0.07&	0.01&	0.12&	0.02&	0.25\\
500&	0.00&	0.17&	0.00&	0.13&	0.02&	0.25&	0.02&	0.46&	0.05&	1.04\\
1000&	0.00&	0.23&	0.01&	0.36&	0.03&	0.57&	0.05&	1.01&	0.10&	2.46\\
2000&	0.01&	0.40&	0.02&	0.55&	0.05&	1.05&	0.11&	2.02&	0.23&	4.22\\
5000&	0.03&	0.78&	0.05&	1.24&	0.14&	2.62&	0.30&	5.16&	0.62&	10.90\\
10000&	0.05&	1.94&	0.11&	2.84&	0.30&	5.72&	0.64&	13.00&	1.35&	23.59\\
20000&	0.12&	5.65&	0.24&	7.68&	0.64&	18.57&	1.37&	29.03&	2.90&	52.02\\
50000&	0.32&	27.34&	0.64&	46.55&	1.75&	61.71&	3.75&	88.38&	8.13&	150.42\\
100000&	0.67&	122.72&	1.38&	136.70&	3.89&	166.49&	8.40&	217.51&	18.26&	344.78\\
200000&	1.46&	420.28&	3.07&	442.55&	8.59&	516.79&	18.86&	643.24&	40.56&	1040.59\\
\hline
\end{tabular}
\end{threeparttable}
\end{table}

\section*{Acknowledgement}

We are grateful to Jiawei Zhang for introducing us this problem.
We thank Bo Jiang and Yi Xu for the discussion and useful suggestions.


\bibliographystyle{model1-num-names}

\end{document}